\documentclass[12pt,a4paper]{amsart}
\usepackage[utf8]{inputenc}
\usepackage[T1]{fontenc}
\usepackage{amsmath, amssymb}
\usepackage{amsfonts,amsthm,amsmath,amssymb,amscd}
\usepackage{graphicx,color}
\usepackage[dvipsnames]{xcolor}
\usepackage{mdframed}
\usepackage{tikz, pgfplots}
\usetikzlibrary{positioning}
\usetikzlibrary{calc}
\pgfplotsset{compat=1.18}
\def\centerarc[#1](#2)(#3:#4:#5)
    {\draw[#1] ($(#2)+({#5*cos(#3)},{#5*sin(#3)})$) arc (#3:#4:#5);}
\usetikzlibrary {arrows.meta}

\usepackage{enumerate}
\theoremstyle{definition}
\newtheorem{dfn}{Definition}

\theoremstyle{plain}
\newtheorem{lem}[dfn]{Lemma}
\newtheorem{thm}[dfn]{Theorem}

\newtheorem{cor}[dfn]{Corollary}

\newcommand{\C}{{\mathbb C}}
\newcommand{\R}{{\mathbb R}}
\newcommand{\Z}{{\mathbb Z}}
\newcommand{\N}{{\mathbb N}}
\renewcommand{\Re}{\operatorname{Re}}
\renewcommand{\Im}{\operatorname{Im}}

\newcommand{\diam}{\operatorname{diam}}

\newcommand{\bd}{\operatorname{bd}}
\newcommand{\ubd}{\operatorname{ubd}}

\newcommand{\La}{\Lambda_\mathbf{c}}
\renewcommand{\c}{\mathbf{c}}
\newcommand{\intt}{\operatorname{int}}

\begin{document}
\begin{abstract}
We study the dynamics of the exponential map on the complex plane. The set $\Lambda_{\c}$ of all points sharing a given itinerary $\c$ is non-empty if and only if $\c$ is an exponentially bounded itinerary. For such itineraries, $\Lambda_{\c}$ also contains a curve of escaping points, and hence its Hausdorff dimension is at least~$1$. We prove that for every exponentially bounded itinerary this dimension is in fact equal to~$1$. In comparison, for certain itineraries, the set $\Lambda_{\c}$ exhibits highly complicated topological structures, such as indecomposable continua.
\end{abstract}

\address{Joanna Horbaczewska, University of Warsaw, ul. Banacha 2, 02-097 Warsaw, Poland}
\email{j.horbaczewska@uw.edu.pl}
\author{Joanna Horbaczewska}

\address{Rados\l{}aw Opoka, University of Warsaw, ul. Banacha 2, 02-097 Warsaw, Poland}
\email{r.opoka@student.uw.edu.pl}
\author{Rados\l{}aw Opoka}

\address{{\L}ukasz Pawelec, Institute of  Mathematical Economics, SGH Warsaw School of Economics, al.~Niepodleg\l{}o\'{s}ci~162, 02-554 Warszawa, Poland}
\email{lpawel@sgh.waw.pl}
\author{\L{}ukasz Pawelec}
\thanks{The research of J. Horbaczewska and R. Opoka was funded in whole or in part by the National Science Centre, Poland, grant no. 2023/51/B/ST1/00946. The research of Ł. Pawelec was funded in whole or in part by the National Science Centre, Poland, grant no. 2023/49/B/ST1/03015.}

\title[Dimension of itineraries for $\exp(z)$]{Dimension of the accumulation set of any hair for the exponential map}
\maketitle

\section{Introduction}

We study the dynamical systems generated by the functions $f_\lambda\colon\C\longrightarrow\C$ defined by
\[f_\lambda(z)=\lambda e^z,\quad \lambda\in\C\setminus\{0\}.\]

The structure of the Julia set for this family depends strongly on the parameter $\lambda$. For example, if $0<\lambda\leqslant\frac{1}{e}$, then the Julia set has a topological structure of the so-called Cantor bouquet and consists of uncountably many curves, each of which can be coded by the behaviour of orbits of points within the curve. This led to the notion of an itinerary (see below for the definition). 

Even when the Julia set is the entire plane (which happens for example when $\lambda>\frac{1}{e}$) it still dissolves into sets with a given itinerary. Such sets may have an interesting topological structure as discussed briefly below, however, in this text we will be interested in its Hausdorff dimension, which is known to behave \emph{paradoxically} for the exponential map -- for classic example, see \cite{Karp}. 

For this purpose, we use the symbolic dynamics defined by the following construction. First, we divide the plane $\C$ into horizontal strips 
\begin{equation}\label{strips2}
P_k=\{z\in\C:(2k-1)\pi<\Im(z)\leqslant(2k+1)\pi\}.
\end{equation}
Observe that every set $P_k$ is mapped bijectively onto $\C\setminus\{0\}$, and the image of the boundary of $P_k$ is the negative real axis.
Then, to every $z\in\C$ we assign a sequence of integers $\c=c_0c_1\ldots$ is such a way that $f^n(z)\in P_{c_n}$ for every $n\in\N$. We call such a sequence \emph{the itinerary of $z$.}
The set of all points with a given itinerary is denoted as $\Lambda_\c$. 
This set can be divided into three disjoint parts further depending on the behaviour of trajectories of points: points with a bounded orbit, points with an unbounded orbit but not going to infinity, and points escaping to infinity:
\[\Lambda_\c = \Lambda_\c^{\bd} \cup  \Lambda_\c^{\ubd} \cup \Lambda_\c^{\infty}.\]

R. Devaney and M. Krych  proved in \cite{DK} that for $\lambda>0$ the set $\Lambda_\c$ is nonempty if and only if the itinerary $\c$ is \emph{exponentially bounded}, i.e. there exists $x\in\R$ such that for every $n\in\N$ we have $2\pi|c_n|\leqslant f_\lambda^n(x)$.
They proved, in fact, that for exponentially bounded itineraries $\c$ the set $\Lambda_\c^\infty$ contains a curve tending to the point at infinity. We call that curve a tail of a hair or the tail of a ray.
This result was generalized by D. Schleicher and J. Zimmer in \cite{SZ}. First, they changed the definition of exponentially bounded itinerary in such a way that the condition in that definition does not depend on $\lambda$: $\c$ is exponentially bounded if 
\[\exists_{A,x>0} \forall_{k\in\N}\hspace{1mm}|c_k|<AF^k(x),\] where the map $F\colon\R\longrightarrow \R$ is defined as $F(t) = e^t - 1$. With that new definition they proved that for every $\lambda\in\C\setminus\{0\}$ the set $\Lambda_\c^\infty$ contains a curve tending to the point at infinity (called the tail of a hair). Moreover they proved that for every tail there exists a maximal curve which extends that tail (we call that extended curve hair) and any escaping point (i.e. $z\in\C$ such that $f_\lambda^n(z)\xrightarrow{n\to\infty}\infty$) is an element of the unique hair, or is a landing point (endpoint) of the unique hair, or maps after finitely many iterations on the hair on which the singular value 0 lies. The third case only occurs when the singular value escapes.
M.~Viana in \cite{VIA} showed that those tails of the hairs (for the complex exponential family) are smooth $C^{\infty}$-curves.

However, the topological structure of $\Lambda_{\c}$ may be much more complicated, as it contains (in most cases is equal to) the closure of the entire hair. The full catalogue of possibilities (depending on $\lambda$ and on $\c$) is still unknown. 
For some itineraries the set $\Lambda_{\c}$ consists only of a hair together with its endpoint, i.e. $\Lambda_\c^{\ubd}=\emptyset$. Such a situation occurs, for instance, for so-called fast itineraries, or for itineraries with finitely many $0$'s as entries (see \cite{SZ}). In the former case the landing point of the hair is necessarily escaping, while in the latter it is not, but can have either a bounded or an unbounded trajectory. 
On the other hand, the set $\Lambda_{\c}$ can be a very complicated topological space.
R.~Devaney proved in \cite{D} that for $\lambda > \frac{1}{e}$ there is a natural compactification of the set (the symmetric upper half-plane part of the set $\Lambda_{000\ldots}$) 
\[
\{z \in \C : \forall_{n \in \N} \quad 0 \leqslant \Im(f_\lambda^n(z)) \leqslant \pi\} \subseteq \Lambda_{00\ldots},
\]
which is an \textit{indecomposable continuum}. Later, R.~Devaney and X.~Jarque proved in \cite{DJ} that for $\lambda > \frac{1}{e}$ and bounded itineraries of the form
\[
s = 0_{n_0} t_0 0_{n_1} t_1 0_{n_2} t_2 \ldots,
\]
where $t_0, t_1, \ldots$ are finite blocks of integers in which at least one term is non-zero, and $0_{n_k}$ is a finite block of zeros of length $n_k$, the set $\Lambda_{\c}$ is an indecomposable continuum, provided that the lengths of the zero blocks are large enough. This is a consequence of the fact that for such itineraries $\Lambda_{\c}$ is the closure of $\Lambda_{\c}^{\infty}$, which, in this case, is a very tangled curve that \textit{accumulates everywhere on itself}. The second-named author found similar phenomena for unbounded itineraries (see \cite{RO}). There are also examples in the general context of $\lambda \in \C$ for which the asymptotic value $0$ is an escaping point (see \cite{REM}). 

For $\lambda = 2\pi i$ (a Misiurewicz parameter) one can obtain an indecomposable continuum as a result of two hairs accumulating on themselves instead of one (see \cite{DJR}); one can also construct an itinerary for which the closure of the hair is a bounded indecomposable continuum in the complex plane, which may contain part of the hair or be disjoint from it -- for instance, it can be a topological sine curve (see \cite{FZ}). In all these cases the constructed itineraries are of a specific block form.

Since for exponentially bounded itineraries $\c$ the set $\Lambda_{\c}$ contains a smooth curve, the Hausdorff dimension of $\Lambda_{\c}$ is at least $1$. It seems that in case of complicated topological structures this lower bound should not be equal the dimension, however, this is not the case. A.~Zdunik with the third-named author proved in \cite{PZ} that for bounded itineraries the Hausdorff dimension of $\Lambda_{\c}$ is at most 1, meaning that (for bounded itineraries $\c$) no matter how simple or how complicated $\Lambda_{\c}$ is topologically, its Hausdorff dimension equals $1$.  We extend this result to all exponentially bounded itineraries, and prove the following.

\begin{thm}\label{thm1}

For $\lambda =1$, and for any itinerary $\c$ we have 
\[\dim_H\left(\Lambda_\c\right)\leqslant 1.\]
\end{thm}
This result together with aforementioned knowledge that any non-empty set $\Lambda_{\c}$ contains a curve leads to the following.
\begin{cor}
For the map $e^z$, whenever $\Lambda_\c$ is non-empty, which happens if and only if the itinerary is exponentially bounded,
\[\dim_H\left(\Lambda_\c\right)= 1.\]
\end{cor}

In this paper we limit ourselves to the case $\lambda =1$. It seems that all the techniques work for any $\lambda>\frac1e$, but we did not want to obfuscate the proof with unwanted technical difficulties. As some of the estimates in the course of the proof are subtle, we decided to prove with extra attention to the constants and estimates, which unfortunately makes the proof look more technical then it should.

It should be noted that unfortunately our method does not work for $\lambda \notin \R$, notably for a very interesting case of $\lambda = 2\pi i$.

\section{Definitions and notations}

We use the notation $\N =\{0,1,2,\ldots\}$.
For $M\geqslant 0$ we define closed half-planes
\[H_M^{+}=\{z:\Re z \geqslant M\} \mbox{\qquad and \qquad} H_M^{-}=\{z:\Re z \leqslant M\},\]
open half-planes
\[\tilde{H}_M^{+}=\{z:\Re z > M\} \mbox{\qquad and \qquad} \tilde{H}_M^{-}=\{z: \Re z < M\},\]
and we also define $H_M^{\text{mid}} = \{z:|\Re(z)| < M\}$.
We denote a rectangle $R_k^p = \{z\in P_p: \Re z \in[k,k+1)\}$. We denote by $B(z,r)$ the open disk with the centre at $z\in\C$ and the radius $r>0$.

Throughout the text we will denote by $f_k^{-1}$ the branch of $f^{-1}$ into the strip $P_k$. We define the \textit{shift-map} $\sigma$ on itineraries such that 
\[\sigma(c_0c_1c_2\ldots) = c_1c_2\ldots,\]
and therefore if $z$ has an itinerary $\c$ then $f(z)$ has an itinerary $\sigma(\c)$.

We will use the following version of the Koebe Distortion Theorem (see \cite{POM}).
\begin{thm}\label{thm:Koebe}
For any map holomorphic injective map $\varphi\colon B(a,r) \to \C$ we have

\[\frac{1-\theta}{(1+\theta)^3} |\varphi'(a)|\leqslant |\varphi'(z)| \leqslant  \frac{1+\theta}{(1-\theta)^3}|\varphi'(a)|,\]
where $\theta = \frac{|z-a|}{r}$.
\end{thm}

\section{Proof of the Main Theorem}
Start by observing that $\Lambda^{\bd}_\c$ consists of at most 2 points. 
For unbounded itineraries this set is trivially empty. A bounded itinerary not ending in a sequence of $0$'s has exactly one such point by \cite[Theorem 6.1]{DJ}. For an itinerary composed of only $0$'s there exist two points with bounded orbit by \cite[Theorem 4.9]{D} (those are repelling fixed points), therefore for every itinerary which ends in only $0$'s there exists two such points. 

Thus, when studying the Hausdorff dimension we may assume that every point from $\La$ has an unbounded trajectory.

To get our result we fix $\delta>0$ and show that there exists some large $M$ and a family of coverings $\mathcal{F}_{\c}^{n}$ (depending on the itinerary and on $M$) such that 

\begin{align*}
&\sup_{Q\in\mathcal{F}_{\c}^{n}}\diam(Q) \to 0 \mbox{\quad as $n\to +\infty$},\\
&\sum_{Q\in\mathcal{F}_{\c}^{n}} (\diam Q)^{1+\delta} \leqslant 1 \mbox{\quad for all $n\in \N$},
\end{align*}
which proves that $\dim_H(\La) \leqslant 1+\delta$. Letting $\delta$ go to zero finishes the proof.

We will choose our $M$ from the orbit of 0. More precisely, let $M>0$ be such that there exists $N_M\in\N$ satisfying $M=f^{N_M}(0)$. It needs to be big enough to satisfy the following
\begin{equation}\label{ineq:<<1}
\frac{e^{-M\delta}}{1-e^{-\delta}} \ll 1
\end{equation}
for $\delta$ small enough. We will additionally assume that $M> 100$ for some basic estimates to hold. We will also use a constant $\eta>1$. For precision of estimates we have taken $\eta=1.1$. 

From this point on, we assume that $\delta$ and $M$ are fixed. We also assume that $\La$ is non-empty. Throughout the proof, whenever it does not lead to confusion, in notation we will drop the subscript $\c$.

Our coverings will be defined differently depending on the position in the plane -- the far-right part, the far-left part and the middle part. Every part will need different treatment. 

As mentioned in the introduction, A. Zdunik with the third-named author proved a similar result for bounded itineraries -- some methods from that proof will be used also here, but only for the left/right parts. The middle one requires completely new techniques. Let us start with some results about the behaviour of the map in $H_M^{\text{mid}}$.

\subsection{Escape time and escape sets in the middle of the plane}\label{subs:mid}

Recall that we may assume that the orbit of any point $z\in\La$ is unbounded, so $\Re(f^k(z)) \geqslant M$ for some $k$. Therefore, let us define escape time
\[\tau_M(z)= \min\{k\geqslant 0 : f^k(z)\in H_M^+\} < +\infty.\]
Note that if the trajectory of $z \in \tilde{H}_M^{-}$ leaves $\tilde{H}_M^{-}$ it has to do it via  $E_0:=B(0,e^M)\cap H_M^+$. Thus, if $\Re(z) < M$, then the escape time $\tau_M$ can be characterized as
\[\tau_M(z)= \min\{k\geqslant 1 : f^k(z)\in E_0\}.\]
Since we fixed $M$ we write $\tau(z)$ instead of $\tau_M(z)$ for simplicity. Let us divide $\tilde{H}_M^{-}$ into sets with a given escape time. 
\begin{dfn}
For $k >0$ let $E_k$ be defined as
\[E_k = \{z\in \tilde{H}_M^{-} : \tau(z) = k \}.\] 
\end{dfn}
Note that for any $k> 0$ we have $f(E_k) \subseteq E_{k-1}$ (using the definition of $E_0$ above). We also have to partition every $E_k$ depending on the itineraries of the points in it.

\begin{dfn}
For $k\geqslant 1$ and the finite sequence of integers $a_0, a_1,\ldots, a_{k-1}$ let $E_k^{a_0, a_1,\ldots, a_{k-1}}$ be the set of all $z\in E_k$ such that first $k$ entries of the itinerary of $z$ are $a_0, a_1,\ldots, a_{k-1}$, i.e. 
\[E_k^{a_0, a_1,\ldots, a_{k-1}} = \{ z\in \tilde{H}_M^{-} : f^j(z)\in P_{a_j}\cap \tilde{H}_M^{-} \text{ for } 0\leqslant j \leqslant k-1,\;f^k(z)\in H_M^+\}.\]
\end{dfn}

For every $k\geqslant 1$, every sequence $a_0, a_1,\ldots, a_{k-1}$, and every $z\in E_k^{a_0, a_1,\ldots, a_{k-1}}$ we have $f^k(z)\in E_0$. Therefore, we can characterize sets $E_k^{a_0, a_1,\ldots, a_{k-1}}$ by an inductive scheme
\[E_1^{a_0} = f_{a_0}^{-1}(E_0)\setminus E_0,\]
recall that $f_{a_0}^{-1}$ is a branch of $f^{-1}$ with values in $P_{a_0}$, and
\[E_{k+1}^{a_0, a_1,\ldots, a_{k}} = f_{a_0}^{-1}(E_{k}^{a_1, a_2\ldots, a_{k}})\setminus E_0,\]
for every $k\geqslant 1$ unless $0\in E_{k}^{a_1, a_2\ldots, a_{k}}$. If $0\in E_{k}^{a_1, a_2\ldots, a_{k}}$ then 
\[E_{k+1}^{a_0, a_1,\ldots, a_{k}} = f_{a_0}^{-1}(E_{k}^{a_1, a_2\ldots, a_{k}}\setminus \{0\})\setminus E_0.\]
Moreover if $0\in E_{k}^{a_1, a_2\ldots, a_{k}}$ then $f^k(0)=M$, i.e. $k=N_M$, and $a_1=a_2=\ldots =a_k = 0$.
Let us note that 
\[f(E_{k+1}^{a_0, a_1,\ldots, a_{k}}) \subseteq E_{k}^{a_1, a_2\ldots, a_{k}}\]
for every $k\geqslant 1$, and $f(E_1^{a_0}) = E_0$.

We need two results concerning the shape/behaviour of the sets above. The first lemma states that the orbit (till it hits $E_0$) of any such set may intersect the negative real line only once. Later in the proof this will limit the number of elements of a covering.

\begin{lem}\label{lem:E_split_only_once}
Let $k\geqslant 1$ and $a_0,a_1,\ldots,a_{k-1}$ be a finite sequence of integers.
There exists at most one $p\in\{0,1,\ldots,k\}$ such that $f^p(E_k^{a_0,a_1,\ldots,a_{k-1}})\cap \R_{-}\neq\varnothing$.
\end{lem}
\begin{proof}
Let $A=E_k^{a_0,a_1,\ldots,a_{k-1}}$. First, let us notice that $f^k(A)\cap \R_{-}=\varnothing$ because $f^k(A)\subseteq E_0$. Now, assume the opposite that there exist $p_1$ and $p_2$, $0\leqslant p_1 < p_2 \leqslant k-1$, such that 
\[f^{p_1}(A)\cap \R_{-}\neq\varnothing \text{ and } f^{p_2}(A)\cap \R_{-}\neq\varnothing.\]
We have $f^{p_2-p_1}(f^{p_1}(A))= f^{p_2}(A)$ therefore $f^{p_2}(A)$ intersect both $\R_{-}$ and $\R_{+}$.
Since $M=f^{N_M}(0)$ then $\tau(z_{+}) < \tau(z_{-})$ for $z_{+}\in\R_{+}$ and $z_{-}\in\R_{-}$.
But for all $z\in f^{p_2}(A)$ we have $\tau(z)=k-p_2$, which gives a contradiction.
\end{proof}
The second lemma says that such partition sets cannot intersect both positive and negative real half-line at the same time and that the intersection has to behave \emph{well-enough}, see Figure \ref{fig:inter} for an illustration.

\begin{lem}\label{lem:inter}
Let us suppose that there exists $k \geqslant 1$ and $a_0,a_1,\ldots,a_{k-1}$ such that 
\[E_k^{a_0,a_1,\ldots,a_{k-1}}\cap \R_{-}\neq\varnothing.\] 
Then $E_k^{a_0,a_1,\ldots,a_{k-1}}\cap \R_{\geqslant 0}=\varnothing$ and $E_k^{a_0,a_1,\ldots,a_{k-1}}\cap \R_{-}$ is a line segment. Moreover, all points of that line segment, except the endpoints, lie in the interior of $E_k^{a_0,a_1,\ldots,a_{k-1}}$.
\end{lem}
\begin{proof}
Let $A=E_k^{a_0,a_1,\ldots,a_{k-1}}$. We have $A\cap \R_{-}\neq\varnothing$.
As in the proof of Lemma \ref{lem:E_split_only_once} points in $\R_{-}$ and points in $\R_{\geqslant 0}$ have different escape times $\tau$, therefore $A\cap \R_{\geqslant 0}=\varnothing$.

To prove that the set $A\cap \R_{-}$ is a line segment it is enough to prove that if $z_1<z_2<z_3$ and $z_1, z_3 \in A\cap \R_{-}$ then $z_2\in A\cap \R_{-}$.
We have 
\[\Re(f^n(z_1)) < \Re(f^n(z_2)) < \Re(f^n(z_3)),\]
for every $n\geqslant 0$. In particular,
\[\Re(f^k(z_1)) < \Re(f^k(z_2)) < \Re(f^k(z_3))\]
thus $f^k(z_2)\in E_0$ and therefore $z\in A$.

Now, suppose that $z$ is an element of obtained line segment which is not an endpoint of that line segment. Then $M<\Re(f^k(z))<e^M$, which means that $f^k(z)\in \intt E_0$. Therefore there exists $\varepsilon>0$ such that $B(z,\varepsilon)\subseteq A$.
\end{proof}

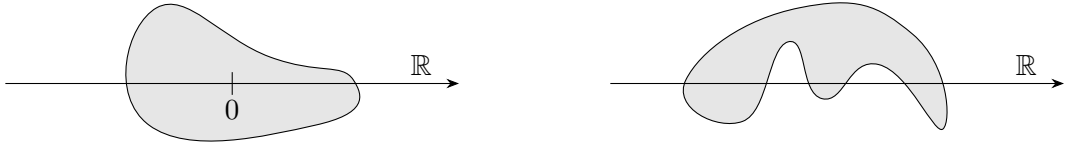
\begin{figure}[h!]\centering
	\begin{tikzpicture}[font=\small]

\begin{scope}
    \draw [black,fill=black!10] plot [smooth cycle, tension=0.9] coordinates {
        (-1,1) (0.5,0.35)(1.6,0.05)(1.1,-0.55)(-1.1,-0.55)
    };
    \draw[-Stealth] (-3,0) -- (3,0);
    \node at (0,-0.35) {$0$};
    \node at (2.5,0.25) {$\R$};
    \draw[-] (0,-0.15) -- (0,0.15);

\end{scope}
 
\begin{scope}[xshift=8cm]
    \draw [black,fill=black!10] plot [smooth cycle, tension=0.9] coordinates { (-2,0.05)(-1.25,-0.5)(-0.65,0.55) (-0.2,-0.2)(0.55,0.25) (1.4,-0.6) (1,0.65) (-0.5,1)
    };
    \draw[-Stealth] (-3,0) -- (3,0);
    \node at (2.5,0.25) {$\R$};
    
\end{scope}
 
\end{tikzpicture}
		\caption{\footnotesize Impossible shapes of $E_k^{a_0, a_1,\ldots, a_{k-1}}$ by Lemma \ref{lem:inter}.}
	\label{fig:inter}
\end{figure}

From the above lemmas it follows that if $C$ is a connected component of $E_k^{a_0,a_1,\ldots,a_{k-1}}$ and for some $p\in\{0,1\ldots, k-1\}$ we have $f^p(C)\cap\R_{-}\neq\varnothing$ then $C$ can be divided into two disjoint parts: the first is mapped with $f^p$ into $f^p(C)\cap \{z:\Im(z) \geqslant 0\}$ and the second is mapped with $f^p$ into $f^p(C)\cap \{z:\Im(z) < 0\}$. We use that to further divide such $E_k^{a_0,a_1,\ldots,a_{k-1}}$ into two parts as presented in the following definition.

\begin{dfn}
Let $k\geqslant 1$ and let $a_0, a_1,\ldots, a_{k-1}$ be the finite sequence of integers. Suppose that there exists $p\in\{0,1\ldots, k-1\}$ such that $f^p(E_k^{a_0,a_1,\ldots,a_{k-1}})\cap\R_{-}\neq\varnothing$.
We define $E_k^{a_0,a_1,\ldots,a_{k-1}}(+)$ to be the union of all connected components of $C$ of $E_k^{a_0,a_1,\ldots,a_{k-1}}$ such that $f^p(C)$ is a subset of upper half-plane $\{z:\Im(z) \geqslant 0\}$. We define $E_k^{a_0,a_1,\ldots,a_{k-1}}(-)$ to be the union of all connected components of $C$ of $E_k^{a_0,a_1,\ldots,a_{k-1}}$ such that $f^p(C)$ is a subset of lower half-plane $\{z:\Im(z) < 0\}$. 
If there is no $p\in\{0,1\ldots, k-1\}$ such that $E_k^{a_0,a_1,\ldots,a_{k-1}}\cap\R_{-}\neq\varnothing$, then we define 
\[E_k^{a_0,a_1,\ldots,a_{k-1}}(+)= E_k^{a_0,a_1,\ldots,a_{k-1}}\]
and 
\[E_k^{a_0,a_1,\ldots,a_{k-1}}(-) = \varnothing.\]
\end{dfn}
With the above definition we have 
\[E_k^{a_0,a_1,\ldots,a_{k-1}}(+)\cup E_k^{a_0,a_1,\ldots,a_{k-1}}(-)= E_k^{a_0,a_1,\ldots,a_{k-1}}\]
and
\[E_k^{a_0,a_1,\ldots,a_{k-1}}(+)\cap E_k^{a_0,a_1,\ldots,a_{k-1}}(-)= \varnothing\]
for every $k\geqslant 1$ and every sequence $a_0,a_1,\ldots,a_{k-1}$.

\subsection{Technical lemmas}
Before we construct the coverings we need a few results concerning the behaviour of the orbit of the map. 

We start with a lemma that says that the trajectory of any point starting in the far-left half plane has to stick \emph{close} to the orbit of 0 for least till it reaches far-right half-plane.

\begin{lem}\label{lem:from_left_to_right}
There exists $M>0$ (it is sufficient to take $M=f^4(0)$) with the following property. For every $x>M$, define
\[
n_x:=\min\left\{n\geqslant 1:f^n(0)\geqslant \frac{x}{3}\right\}.
\]
Then
\[
f^{n_x+1}(\tilde{H}_{-x}^{-})\subseteq B\left(f^{n_x}(0),\frac{\pi}{6}\right)\subseteq H_{\frac{x}{4}}^{+}.
\]
In other words, for any $z$ with $\Re(z)=-x<-M$ the trajectory $f^n(z)$ \emph{detaches} from the orbit of 0 after passing $\frac{x}{4}$.
\end{lem}

\begin{proof}
For $w=f^k(0)+h$ with $|h|\leqslant \frac{\pi}{6}$ we have
$f(w)=f^{k+1}(0)\cdot e^{h}$
and, since $|e^{h}-1|\leqslant |h|\cdot e^{|h|}\leqslant |h|\cdot e^{\frac{\pi}{6}}$,
\begin{equation}\label{eq:rho_est}
|f(w)-f^{k+1}(0)|=f^{k+1}(0)\cdot|e^{h}-1|\leqslant e^{\frac{\pi}{6}}\cdot |h|\cdot f^{k+1}(0).
\end{equation}
Suppose $z\in \tilde{H}_{-x}^{-}$ with $x>0$. Then $f(\tilde{H}_{-x}^{-})\subseteq B(0,e^{-x})$.
Set $\rho_{0}(x)=e^{-x}$ and, recursively,
\[
\rho_{k}(x)=e^{\frac{\pi}{6}}\cdot f^{k}(0)\cdot \rho_{k-1}(x)=e^{k\frac{\pi}{6}-x}\cdot f(0)\cdot f^2(0)\cdot\ldots\cdot f^{k}(0).
\]
Using \eqref{eq:rho_est}, if $\rho_{k}(x)\leqslant \frac{\pi}{6}$, then
\[
f^{k+1}(\tilde{H}_{-x}^{-})\subseteq B\left(f^k(0),\rho_{k}(x)\right) \subseteq B\left(f^k(0),\frac{\pi}{6}\right).
\]
With $x$ large enough we have $f^{n_x}(0) \geqslant \frac{x}{4}+\frac{\pi}{6}$.
Therefore it is enough to prove that $\rho_{n_x}(x)\leqslant \frac{\pi}{6}$.
We have
\[
f(0)\cdot f^2(0)\cdot\ldots\cdot f^k(0)\leqslant (f^{k}(0))^2,
\]
for all $k\geqslant 1$, which can be easily proved by induction. Consequently,
\begin{equation}\label{eq:rho_bound}
\rho_{k}(x)\leqslant e^{k\frac{\pi}{6}-x}\cdot (f^{k}(0))^2
\end{equation}
for all $k\geqslant 1$.
By the definition of $n_x$, we have $f^{n_x}(0) < e^{\frac{x}{3}}$.
Combining this with \eqref{eq:rho_bound} we obtain
\[
\rho_{n_x}(x)< e^{n_x\frac{\pi}{6} -\frac{x}{3}}.
\]
Trivially, for $n\geqslant 4$ we have $e^n < f^{n}(0)$, therefore, with $M$ large enough we have $e^{n_x} < f^{n_x}(0)$. Using $f^{n_x}(0) < e^{\frac{x}{3}}$ again, we obtain $n_x < \frac{x}{3}$. Therefore
\[
\rho_{n_x}(x)< e^{\frac{x}{3}\left(\frac{\pi}{6} -1\right)}.
\]
Since $\frac{\pi}{6} < 1$, then with sufficiently large $x$ (in fact $x>5$ is enough) we have $\rho_{n_x}(x)\leqslant\frac{\pi}{6}$.
\end{proof}

The next lemma is about the distortion of the iterates of $f$.
\begin{lem}\label{lem:distortion_v0}
Let $n\geqslant 1$ and let $A$ be such that for every $m\in\{0,1,\ldots,n\}$ we have $f^m(A)\subseteq B\left(f^m(0),\frac{\pi}{6}\right)$. Let $g=f^n$ on $A$. Then $g$ has a bounded distortion on $A$ and 
\[\frac{\sup_{z\in A}|g'(z)|}{\inf_{z\in A}|g'(z)|}\leqslant \frac{\left(1+\frac{\pi}{6r}\right)^4}{\left(1-\frac{\pi}{6r}\right)^4}\]
for $r=f^n(0)-f^{n-1}(0)$.
\end{lem}
\begin{proof}
Let $\varphi$ be the inverse of $g$. The function $\varphi$ can be defined on $B = B(f^{n}(0),r)$ and is univalent so by Theorem \ref{thm:Koebe} we have
\[\frac{1-\theta}{(1+\theta)^3} |\varphi'(f^{n}(0))|\leqslant |\varphi'(z)| \leqslant  \frac{1+\theta}{(1-\theta)^3}|\varphi'(f^{n}(0))|\]
for every $z\in g(A)$, where $\theta = \frac{\pi}{6r}$. Thus, we have
\[\frac{\sup_{z\in g(A)}|\varphi'(z)|}{\inf_{z\in g(A)}|\varphi'(z)|}\leqslant \frac{(1+\theta)^4}{(1-\theta)^4}.\]
With that we have also
\begin{equation}\label{eq:distortion theta}
\frac{\sup_{z\in A}|g'(z)|}{\inf_{z\in A}|g'(z)|}\leqslant \frac{(1+\theta)^4}{(1-\theta)^4}.
\end{equation}
\end{proof}

The following, rather technical, lemma studies the behaviour of short trajectories near~0.

\begin{lem}\label{lem:eta^4}
If $z$ satisfies $-3\leqslant\Re(z)< \ln(\eta)$, then
\[|(f^4)'(z)|\geqslant \eta^4.\]
\end{lem}

\begin{proof}
We need to prove that for every $z$ such that $-3\leqslant\Re(z)\leqslant \ln(\eta)$ we 
\[|f(z)|\cdot |f^2(z)|\cdot|f^3(z)|\cdot|f^4(z)|\geqslant \eta^4.\]
Using the property $|e^z|=e^{\Re(z)}$ we can prove instead the following inequality
\[\Re(z) +  \Re(f(z))+  \Re(f^2(z))+  \Re(f^3(z))\geqslant 4\ln(\eta).\]
The map $f(z)$ is $2\pi i$-periodic, therefore it is enough to prove that the above inequality is true for $z$ from the set $A=\{w\in\C: -3\leqslant\Re(w)\leqslant \ln(\eta),\; 0\leqslant\Im(w)\leqslant 2\pi\}$. The expression 
\[\Re(z) +  \Re(f(z))+  \Re(f^2(z))+  \Re(f^3(z))\]
is a harmonic function as a sum of real values of holomorphic maps, therefore it is enough to prove that inequality
\[g(z) = \Re(z) +  \Re(f(z))+  \Re(f^2(z))+  \Re(f^3(z))\geqslant 4\ln(\eta)\]
holds for the points on the boundary of $A$. By expressing $g(z)$ as a function $g(x,y)$ with $x=\Re(z)$ and $y=\Im(z)$ we define
\[h(x) = g(x,0) \text{ and } h_a(x) = g(a,x).\]
It can be checked by computation that the minimal value of the function $h(x)$ on the interval $[-3,\ln(\eta)]$ and the minimal values of the functions $h_a(x)$ for $a\in\{-3,\ln(\eta)\}$ on the interval $[0,2\pi]$ are greater than $4\ln(\eta)$.



\begin{figure}[ht]
\centerline{\includegraphics[scale=0.25]{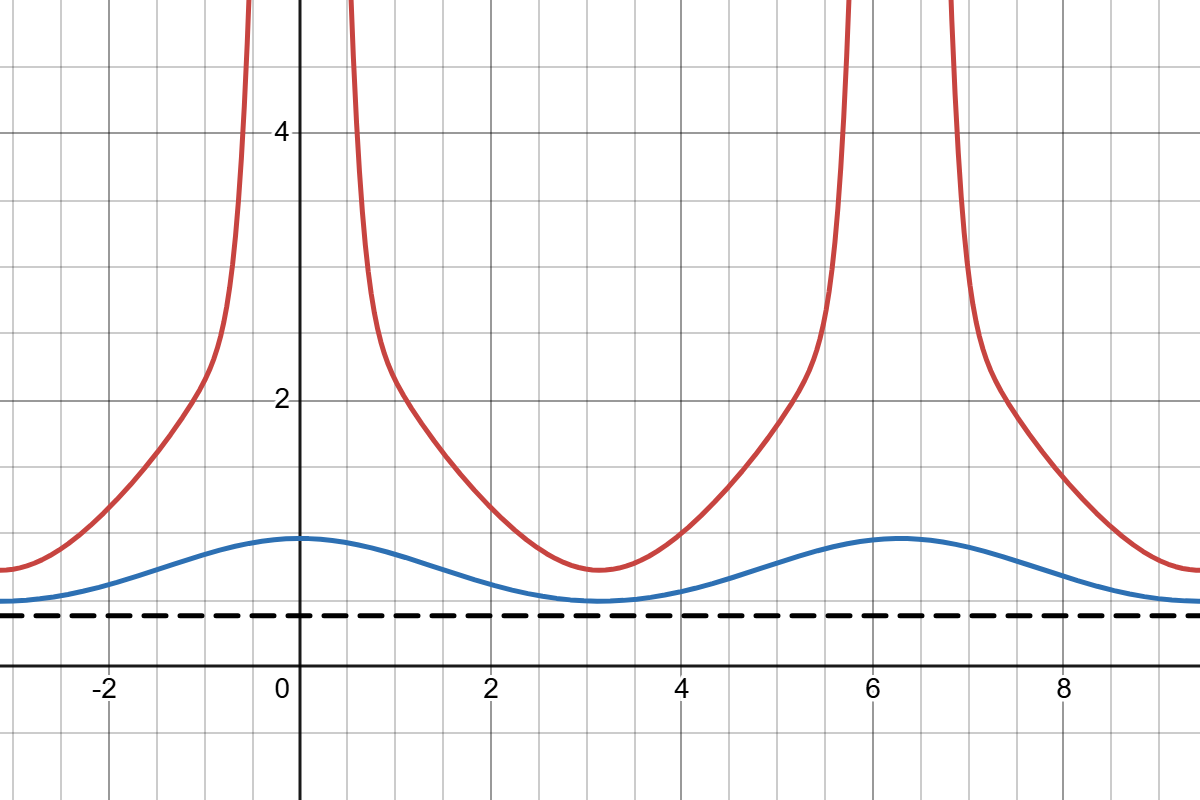}}
\caption{\footnotesize Graphs of the functions $h_{-3}$ (blue) and $h_{\ln(\eta)}$ (red) together with the horizontal line at the height $4\ln(\eta)$ (black).}
\label{fig1}
\end{figure}
\end{proof}

The next result says that whenever we leave left half-plane $\tilde{H}_M^{-}$ the derivative is large.
\begin{lem}\label{lem:eta^k}
For any $z$ with $-e^{M}\leqslant\Re(z)<M$ there exists $j\in\{0,1,2\}$ such that
\[|(f^{\tau(z)+j})'(z)|\geqslant \eta ^{\tau(z)+j} \text{ and } \Re(f^{\tau(z)+j}(z))\geqslant M.\]
\end{lem}
\begin{proof}
Put $\tau=\tau(z)$. We will divide the trajectory $z$, $f(z)$, \ldots, $f^\tau(z)$ into sectors such that on every sector the cumulative derivative grows at least geometrically, i.e. within a sector $f^\alpha(z)$ \ldots $f^{\alpha+p-1}(z)$ we have
\[\prod_{i=\alpha}^{\alpha+p-1} |f'(f^i(z))| \geqslant \eta^{p}.\]
Collecting the derivative over all sectors give the required result.

There are three types of sectors of different lengths depending on the position of $f^i(z)$ in the plane. We construct the partition step by step, i.e. we start with $z$, the first element of the trajectory, and construct the sector of required length, then we take the next element $f^k(z)$ for some $k$ and construct the next sector. We do this till we exhaust the entire trajectory (till $f^\tau(z)$).

Denote by $u$ the starting element of the sector (for the first sector $u=z$). There are three possibilities.

First, $\ln(\eta)\leqslant\Re(u)$ and therefore $f(u)$ lands outside $B(0,\eta)$ so $|f'(u)|>\eta$. Such a sector consist of one element only.

Second, $-3\leqslant\Re(u)<\ln(\eta)$ and therefore $f(w)$ lands in the annulus $B(0,\eta)\setminus B(0,e^{-3})$. Then by Lemma \ref{lem:eta^4} we have \[|(f^4)'(u)|\geqslant \eta^4.\] Such a sector consist of exactly 4 elements.

Third, $\Re(u)< -3$. The length of the sector is $t+1$, a number depending on $\Re(u)$. Let \[R=\{w\in\C: |\Re(w)-\Re(u)|<\frac{1}{2},\; |\Im(w)-\Im(u)|<\pi\}\] 
and
\begin{align*}
t &  = 1+\max\left\{n\in\N: f^n(R)\subseteq B\left(f^{n-1}(0),\frac{\pi}{6}\right)\right\}\\
&=\min\left\{n\in\N: f^n(R)\cap \left(\C\setminus B\left(f^{n-1}(0),\frac{\pi}{6}\right)\right)\neq\varnothing\right\}.
\end{align*}
First, let us prove that $t\geqslant 4$. It is enough to show that if $H=\tilde{H}_{\frac{5}{2}}^{-}$ then $f(H)\subseteq B(0,\frac{\pi}{6})$ and $f^2(H)\subseteq B(1,\frac{\pi}{6})$ and $f^3(H)\subseteq B(e,\frac{\pi}{6})$.
We have $f(H)=B(0,r_0)$ for 
\[r_0=e^{-\frac{5}{2}} < 0.09 < \tfrac{\pi}{6}.\] 
So $f^2(H)\subseteq B(1,r_1)$, where $r_1 = r_0\cdot e^{r_0}$, because $e^{r_0}$ is the maximum value of $|f'|$ on $B(0,r_0)$. We have 
\[r_1 < 0.09\cdot e^{0.09} < 0.1 < \tfrac{\pi}{6}.\]
Using the same reasoning we get that $f^3(H)\subseteq B(e,r_2)$, where 
\[r_2 < 0.1\cdot e^{1+0.1} < 0.31<\tfrac{\pi}{6}.\]
Now, we estimate $|(f^t)'(z)|$. We have $\diam(f^t(R))\geqslant \frac{\pi}{6}$. Therefore, there exists $\xi\in R$ such that $|(f^t)'(\xi)| \geqslant \frac{\pi}{6\diam(R)} \geqslant \frac{2}{25}$. Let $r=f^2(0)-f(0)=e-1$. By Lemma \ref{lem:distortion_v0} we have
\[
\frac{\sup_{z\in f(R)}|(f^2)'(z)|}{\inf_{z\in f(R)}|(f^2)'(z)|}\leqslant\frac{\left(1+\frac{\pi}{6r}\right)^4}{\left(1-\frac{\pi}{6r}\right)^4} \leqslant \frac{25}{2}.
\]
Which gives
\begin{align}\label{eq:distortion_of_composition_v0}
\frac{\sup_{z\in R}|F'(z)|}{\inf_{z\in R}|F'(z)|} &  \leqslant \frac{\sup_{z\in  R}|f'(z)|\cdot\sup_{z\in f(R)}|(f^2)'(z)|\cdot\sup_{z\in f^3(R)}|f'(z)|}{\sup_{z\in  R}|f'(z)|\cdot\inf_{z\in f(R)}|(f^2)'(z)|\cdot\inf_{z\in f^3(R)}|f'(z)|}\nonumber \\
& \leqslant e\cdot\frac{25}{2}\cdot\frac{e^{r+\frac{\pi}{6}}}{e^{r-\frac{\pi}{6}}} = \frac{25}{2}e^{\frac{\pi}{3}+1} < 97.
\end{align}
In the penultimate inequality we used the fact that $f^3(R)\subseteq B\left(f^3(0),\frac{\pi}{6}\right)$.
So, for $t\geqslant 4$, we get
\[|(f^t)'(z)| \geqslant \tfrac{1}{D}|(f^t)'(\xi)|\geqslant \tfrac{2}{25\,\cdot\, 97} = \tfrac{2}{2425} .\]
Let $A = \frac{2}{2425}$.
Thus, we have
\[|(f^{t+1})'(z)| \geqslant A \inf_{w\in R}|f'(f^t(w))|.\]
We have $f^{t-1}(R)\subseteq B(f^{t-2}(0),\frac{\pi}{6})$ and thus $f^{t}(R)$ is a subset of $H_r^+$ for 
\[r = \sin\left(\tfrac{\pi}{6}\right)e^{f^{t-2}(0)-\frac{\pi}{6}} = \tfrac{\sqrt{3}}{2}e^{-\frac{\pi}{6}}f^{t-1}(0),\]
leading to 
\[|(f^{t+1})'(z)|\geqslant A\exp\left(\tfrac{\sqrt{3}}{2}e^{-\frac{\pi}{6}}f^{t-1}(0)\right) \geqslant A\left(f^t(0)\right)^{\frac{1}{2}}.\]
Now, we want to prove (by induction) that for every $t\geqslant 4$ we have

\[A\left(f^t(0)\right)^{\frac{1}{2}} \geqslant \eta^{t+1}\]
with $\eta=1.1$. Denote $c=\frac{1}{2}$. 
For $t=4$ we check that
\[A(e^{e^e})^{c} = \frac{2}{2425}\left(e^{e^e}\right)^{c} > \eta^5.\]
We have
\begin{align*}
A(f^{t+1}(0))^c & = A\exp({f^{t}(0)})^c = A\exp(c\cdot f^{t}(0)) \\ 
& = A\exp\left(c\cdot(f^{t}(0)^c)^{\frac{1}{c}}\right) = A\exp\left(c\cdot\left(\frac{Af^{t}(0)^c}{A}\right)^{\frac{1}{c}}\right).
\end{align*}
Therefore, for the induction step, it is enough to prove that
\[A\exp\left(c\cdot \left(\frac{x}{A}\right)^{\frac{1}{c}}\right) \geqslant \eta x\]
for $x > 1$, which can be easily checked to be true.

Thus we successfully divided the trajectory of $z$ into parts such that on every part the derivative is at least grows geometrically with constant $\eta$. 

The only remaining question is whether the described procedure aligns correctly with the end of our finite trajectory. Recall that, by definition, $f^\tau(z)\in E_0$. Thus the last sector may not be of the second type (as $M>100$). If the last sector is of the first type, then it is of length one and aligns perfectly with our trajectory. 

However, if the last defined sector is of the third type, it may happen that its end will occur after $\tau$. 
But in such a situation, by Lemma \ref{lem:from_left_to_right} and the definition of $t$ the trajectory leads far to the right. How far can it go? Our trajectory started $-M<\Re(z)<M$ and it cannot leave $-e^{M}<\Re(z)<M$ before hitting $E_0$. Thus, the \emph{worst case scenario} (regarding how far can the point escape to the right) is when $\Re(u)=-e^{M}$ ($u$ is the starting element of the sector). Then its trajectory reaches $\Re(f^{t+1}(u))>\frac{\sqrt{3}}{2}e^{e^M/4}$. On the other hand $\Re(f^{t}(u))$ cannot be bigger then $e^M$, because the cumulative derivative would be too big for the diameter of the image of the rectangle to be smaller than $\frac{\pi}{6}$ -- compare the definition of $t$.

But any point in $E_0$ going to the right (and following $\frac{\pi}{6}$ close the orbit of $0$) has to enter $\Re(f^{t+1}(u))>\frac{\sqrt{3}}{2}e^{e^M/4}$ in two steps.
 This shows that $j\leqslant 2$ and finishes the proof.

\end{proof}

\subsection{Construction of the first level covering}
From now on, we assume that $M$ is large enough for all the lemmas above to hold and that $\delta<1$. Further lower bounds on $M$ will be imposed later in the paper. We always assume that $\c$ is exponentially bounded. 

We are now ready to construct the family of coverings. We will start by defining zeroth and first level covering and afterwards define the next levels by induction. Note that the coverings will depend on $M$. We have
\[\La \subseteq \sum_{k\in \Z} R_k^{c_0}.\]
Thus, for a given itinerary $\c$, as the zeroth level covering we define
\[\mathcal{F}^0_\c = \{R_k^{c_0}:k\in \Z\}.\]
Construction of the first level (in fact, all levels) covering $\mathcal{F}^1_{\c}$ depends on the location in the plane. Thus, we will define three subfamilies $\mathcal{R}^1_{\c}$, $\mathcal{L}^1_{\c}$, $\mathcal{M}^1_{\c}$ forming $\mathcal{F}^1_\c$, i.e. $\mathcal{F}^1_{\c} = \mathcal{R}^1_{\c} \cup \mathcal{L}^1_{\c} \cup \mathcal{M}^1_{\c}$
and they are covering of sets $\La\cap H_M^{+}$, $\La\cap H_{-M}^{-}$ and $\La\cap H_M^{\text{mid}}$ respectively.

The first subfamily, $\mathcal{R}^1_{\c}$, is constructed with the following idea: take any rectangle $R_k^p$ such that $R_k^p\cap \tilde{H}_M^{+}\neq\varnothing$, its image is a large annulus intersecting some of the rectangles from $\mathcal{F}_{\sigma(\c)}^0$. Pulling back those rectangles with the function $f$ gives the covering. More precisely, the family is defined as 
\begin{equation}
\mathcal{R}^1_{\c} = \{R_k^{c_0}\cap f^{-1}(R_l^{c_1}): f(R_k^{c_0})\cap R_l^{c_1}\neq\varnothing,\; k\geqslant \lfloor M\rfloor,\; l\in\Z \}.
\end{equation}

\begin{figure}[ht]
    \centering
    \begin{tikzpicture}[font=\small]
        \draw [black,fill=black!10] (2.5,1) rectangle (3,1.75);
        \draw node at (2.8,0.6) {$R_k^{c_0}$};
        \draw node at (4.9,1) {$f(R_k^{c_0})$};
        
        \centerarc[solid](0,0)(-5:185:3.9)
        \centerarc[dashed](0,0)(-15:-5:3.9)
        \centerarc[dashed](0,0)(185:195:3.9)
        
        \centerarc[solid](0,0)(-5:185:6.1)
        \centerarc[dashed](0,0)(-10:-5:6.1)
        \centerarc[dashed](0,0)(185:190:6.1)

        \draw[dashed] (-7,0)--(7,0);
        \draw[solid] (0,0.1)--(0,-0.1);
        \draw node at (0,-0.35) {$0$};
        \draw[dashed] (1.8,6.25)--(1.8,-0.1);
        \draw[solid] (1.8,0.1)--(1.8,-0.1);
        \draw node at (1.8,-0.35) {$M$};

        \draw [red] (-0.5,3.25) -- (-0.5,4);
        \draw [red] (-1,3.25) -- (-1,4);
        \draw [red] (-1.5,3.25) -- (-1.5,4);
        \draw [red] (-2,3.25) -- (-2,4);
        
        \draw [red] (-2.5,3.25) -- (-2.5,4);
        \draw [red] (-3,3.25) -- (-3,4);
        \draw [red] (-3.5,3.25) -- (-3.5,4);
        \draw [red] (-4,3.25) -- (-4,4);
        \draw [red] (-4.5,3.25) -- (-4.5,4);
        \draw [red] (-5,3.25) -- (-5,4);
        \draw [red] (-5.5,3.25) -- (-5.5,4);
        
        \draw [red] (0,3.25) -- (0,4);
        
        \draw [red] (0.5,3.25) -- (0.5,4);
        \draw [red] (1,3.25) -- (1,4);
        \draw [red] (1.5,3.25) -- (1.5,4);
        \draw [red] (2,3.25) -- (2,4);
        
        \draw [red] (2.5,3.25) -- (2.5,4);
        \draw [red] (3,3.25) -- (3,4);
        
        \draw [red] (4.5,3.25) -- (4.5,4);
        \draw [red] (5,3.25) -- (5,4);
        \draw [red] (5.5,3.25) -- (5.5,4);

        \fill [red!10] (3.5,3.25) rectangle (4,4);
        \draw [red] (3.5,3.25) rectangle (4,4);
        
        \draw [red] (-5.5,3.25) -- (3.5,3.25);
        \draw [red] (4,3.25) -- (5.5,3.25);
        \draw [red] (-5.5,4) -- (3.5,4);
        \draw [red] (4,4) -- (5.5,4);

        \draw[red] plot [smooth cycle] coordinates {(2.7,1.5) (2.85,1.55) (2.8, 1.25) (2.7,1.3)};
        \draw[fill=red, opacity=0.1] plot [smooth cycle] coordinates {(2.7,1.5) (2.85,1.55) (2.8, 1.25) (2.7,1.3)};
        
        \draw[-Stealth] (3.75,3.5) to[bend left] (2.8,1.4);
        \draw node at (4.1,2.5) {$f_{c_0}^{-1}$};

        \draw [red] node at (5.35,4.3) {$R_l^{c_1}$};

    \end{tikzpicture}
    \caption{\footnotesize Construction of the covering $\mathcal{R}^1_{\c}$.}
    \label{fig:covering R}
\end{figure}
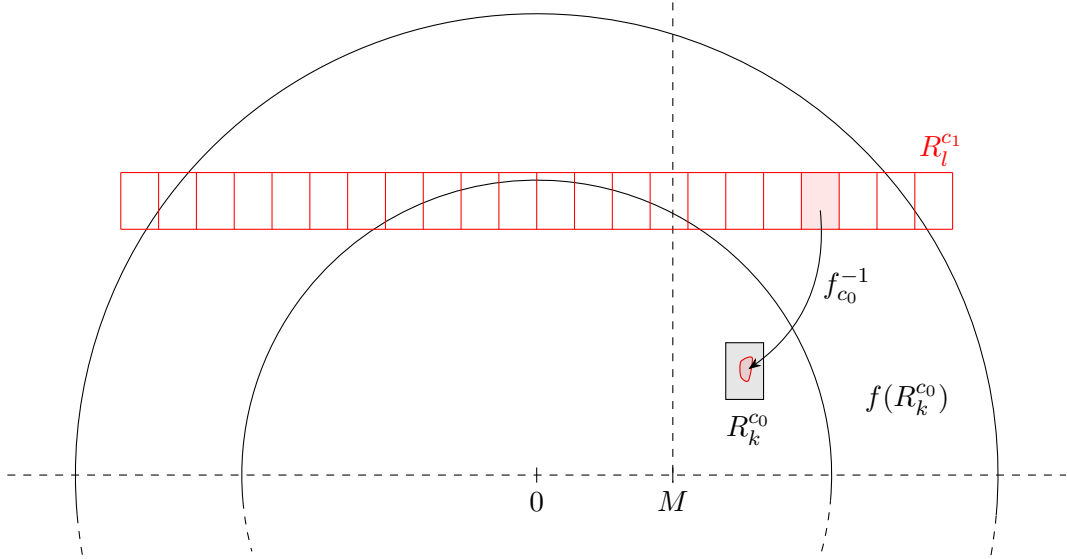

For the subfamily $\mathcal{L}_{\c}^1$, covering far-left part of $\La$,  let us consider rectangle $R_{-k}^{c_0}$ for $k\geqslant \lfloor M\rfloor + 1$. The image of that rectangle is an annulus centred at $0$ with radii $e^{-k}$ and $e^{-k+1}$. Let 
\[t_k=\min\left\{n\in\N: \diam(f^n(B(0,e^{-k+1})))\geqslant \tfrac{\pi}{6}\right\}.\] 
Since $J(f)=\C$, such $t_k$ is well defined. Let us observe that, as in the proof of the Lemma~\ref{lem:eta^k}, since $f^{t_k}(R_k^{c_0})$ is a subset of the ball $B\left(f^{t_k-1}(0),\frac{\pi}{6}\right)$, then $f^{t_k+1}(R_k^{c_0})$ is a subset of $H_{r_k}^{+}$ for 
\[r_k = \sin\left(\tfrac{\pi}{6}\right)e^{f^{t_k-1}(0)-\frac{\pi}{6}} = \tfrac{\sqrt{3}}{2}e^{-\frac{\pi}{6}}f^{t_k}(0).\]
By Lemma \ref{lem:from_left_to_right}, with $M$ large enough, we have $f^{t_k-1}(0)-\frac{\pi}{6} \geqslant \frac{k}{3}$. Therefore, assuming that $M$ is large enough again, we have
\begin{equation}\label{eq:r_k>k}
r_k>k\geqslant \lfloor M\rfloor+1.
\end{equation}
Let us consider branch of $f^{-t_k-1}$ which corresponds to the finite sequence $c_0c_1\ldots c_{t_k}$. Therefore, to obtained the covering $\mathcal{L}_{\c}^1$ we can pull back, with that branch of $f^{-t_k-1}$, all elements of $\mathcal{R}_{\sigma^{t_k+1}(\c)}^1$ intersecting $f^{t_k+1}(R_{-k}^{c_0})$. More precisely,
\begin{equation}
\mathcal{L}^1_{\c} = \{R_{-k}^{c_0}\cap f^{-t_k-1}(S): S\in\mathcal{R}_{\sigma^{t_k+1}(\c)}^1,\; f^{t_k+1}(R_k^{c_0})\cap S\neq\varnothing,\; k\geqslant \lfloor M\rfloor+1,\; l\in\Z \}.
\end{equation}
\begin{figure}[ht]
    \centering
    \begin{tikzpicture}[font=\small]
    
        \draw [fill=black!10] (-5,1.75) rectangle (-4.5,2.5);
        
        \draw node at (-5.5,2.1) {$R_{-k}^{c_0}$};

        \draw[dashed] (-3.65,3)--(-3.65,-0.1);
        \draw[solid] (-3.65,0.1)--(-3.65,-0.1);
        \draw node at (-3.65,-0.35) {$-M$};
        \draw[dashed] (3.65,3)--(3.65,-0.1);
        \draw[solid] (3.65,0.1)--(3.65,-0.1);
        \draw node at (3.75,-0.35) {$M$};

        \draw [black,fill=black!10] (0,0) circle (0.3);
        \draw [black,fill=white] (0,0) circle (0.12);

        \draw [black,fill=black!10] plot [smooth cycle] coordinates {(5,0.5) (5.25,2) (6.5, 1.8) (6.5,0.2) (5,-1.25) (4.5,-0.2)};
        
        \draw [black,fill=black!10] plot [smooth cycle] coordinates {(2.65,0) (3.1,0.4) (3.3, 0.1) (3,-0.4)};

        \draw[dashed,-Stealth] (-7,0)--(7,0);
        \draw[dashed,Stealth-] (0,3.75)--(0,-2.75);

        \draw[-Stealth] (0.3,0.3) to [bend left] (1,0.3);
        \draw node at (1.4, 0.2) {$...$};
        \draw[-Stealth] (1.75,0.3) to [bend left] (2.45,0.3);
        \draw node at (0.65, 0.7) {$f$};
        \draw node at (2.1, 0.7) {$f$};
        \draw[-Stealth] (3.75,0.3) to [bend left] (4.45,0.3);
        \draw node at (4.1, 0.7) {$f$};

        \draw [blue] (3,0) circle (0.47);
        \draw [blue,fill=blue] (3,0) circle (0.05);
        \draw [blue] node at (2, -2) {$B(f^{t_k-1}(0),\frac{\pi}{6})$};
        \draw[blue,-Stealth] (1.9,-1.6) to (2.65,-0.55);

        \draw node at (6.7, -1) {$f^{t_k+1}(R_{-k}^{c_0})$};

        \draw[-Stealth] (-4.25,2) to [bend left] (-0.4,0.3);
        \draw node at (-2,2) {$f$};

        \draw [red] (5,1) rectangle (5.5,1.75);
        \draw [fill=red, opacity=0.1] (5,1) rectangle (5.5,1.75);
        \draw [red] (6,1) -- (6,1.75);
        \draw [red] (6.5,1) -- (6.5,1.75);
        \draw [red] (7,1) -- (7,1.75);

        \draw [red] (5.5,1) -- (7,1);
        \draw [red] (5.5,1.75) -- (7,1.75);

        \draw[red] plot [smooth cycle] coordinates {(-4.98,2.3) (-4.98,2.1) (-4.75, 2) (-4.65,2.1) (-4.75,2.35)};
        \draw[fill=red, opacity=0.1] plot [smooth cycle] coordinates {(-4.98,2.3) (-4.98,2.1) (-4.75, 2) (-4.65,2.1) (-4.75,2.35)};

        \draw[-Stealth] (5.25,1.5) to [bend right] (-4.75,2.25);
        \draw node at (1.5,3.6) {$f^{-t_k-1}$};
        \draw [black] (-5,1.75) rectangle (-4.5,2.5);

        \draw [red] node at (7.25,2.15) {$R_l^{c_{t_k+1}}$};

    \end{tikzpicture}
    \caption{\footnotesize Construction of the covering $\mathcal{L}^1_{\c}$.}
    \label{fig:covering L}
\end{figure}

It remains to construct the \emph{middle} subfamily $\mathcal{M}_{\c}^1$. Let us begin by recalling the definition of $\tau$. For $z\in H_M^{\text{mid}}$ we have 
\[\tau(z)= \min\{k\geqslant 1 : f^k(z)\in H_M^+\} = \min\{k\geqslant 1 : f^k(z)\in E_0\}.\]
Denote by $\beta(\c,k)$ the finite sequence of integers of first $k$ entries of itinerary $\c$, i.e. $c_0,c_1,\ldots,c_{k-1}$. We will use the sets $E_k$ defined in Subsection \ref{subs:mid}. Those sets were defined on a half-plane $\Re(z)<M$, but we will only use them in the middle of the plane. Using the notions from Subsection \ref{subs:mid} we know that the family 
\[\{E_k^{\beta(\c,k)}\cap H_M^{\text{mid}}\}_{k\geqslant 1} = \{E_k^{\beta(\c,k)}(+)\cap H_M^{\text{mid}}\}_{k\geqslant 1} \cup \{E_k^{\beta(\c,k)}(-)\cap H_M^{\text{mid}}\}_{k\geqslant 1}\]
is a covering of $\La\cap H_M^{\text{mid}}$. 

To simplify notation we will always (only) consider the sets $E_k$ intersected with $H_M^{\text{mid}}$, i.e. instead of writing $E_k^{\beta(\c,k)}(\psi)\cap H_M^{\text{mid}}$ we put just $E_k^{\beta(\c,k)}(\psi)$.

According to Lemma $\ref{lem:eta^k}$ for every $k\geqslant 1$ and every symbol $\psi\in\{+,-\}$ we have
\[E_k^{\beta(\c,k)}(\psi) = \bigcup_{j=0}^{2} E_k^{\beta(\c,k)}(\psi, j),\]
where $E_k^{\beta(\c,k)}(\psi, j)=\{z\in E_k^{\beta(\c,k)}(\psi): |(f^{k+j})'(z)|\geqslant \eta^{k+j}\}$ for $j\in\{0,1,2\}$. 
Now, we define covering $\mathcal{M}_{\c}^1$ as 
\begin{align*}
\mathcal{M}_{\c}^1 = \{ f^{-k-j}(S)\cap E_k^{\beta(\c,k)}(\psi, j): & \; S\in\mathcal{R}_{\sigma^{k+j}(\c)}^1,\; f^{k+j}( E_k^{\beta(\c,k)}(\psi, j))\cap S\neq\varnothing, \\
& \; k\geqslant 1,\; j\in\{0,1,2\},\; \psi\in\{+,-\}  \},
\end{align*}
where again $f^{-k-j}$ is a branch which corresponds to $c_0,c_1,\ldots,c_{k+j-1}$. 

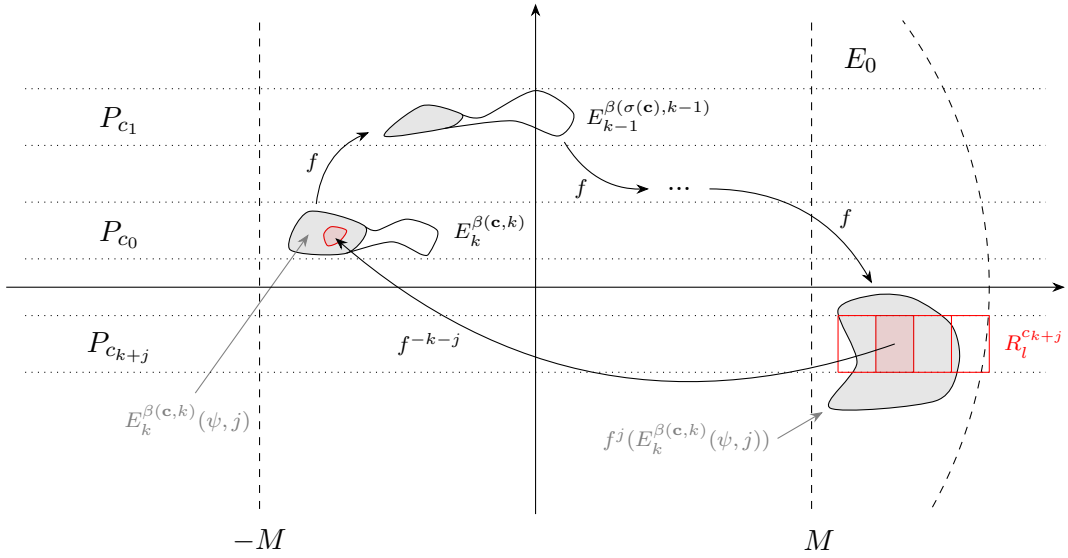
\begin{figure}[ht]
    \centering
    \begin{tikzpicture}[font=\small]

        \draw[dashed] (-3.65,3.5)--(-3.65,-3);
        \draw node at (-3.65,-3.35) {$-M$};
        \draw[dashed] (3.65,3.5)--(3.65,-3);
        \draw node at (3.75,-3.35) {$M$};
        
        \draw [black,fill=black!10] plot [smooth cycle] coordinates {(4,-0.3) (4.5,-0.1) (5, -0.15) (5.5,-0.5) (5.6,-1) (5.3,-1.5) (3.9,-1.6) (4.25,-1)};

        \draw [black,fill=black!10] plot [smooth cycle] coordinates {(-3.25,0.5) (-3,1) (-2.25, 0.8)(-2.5,0.45)};
        \draw [black] plot [smooth] coordinates {(-2.25, 0.8) (-2.1,0.75) (-1.7, 0.9) (-1.3,0.75) (-1.4,0.46) (-1.5, 0.43) (-1.9,0.6)(-2.5,0.45)};
        \draw node at (-0.6,0.75) {\tiny $E_k^{\beta(\c,k)}$};
        \draw [black!50] node at (-4.6,-1.75) {\tiny $E_k^{\beta(\c,k)}(\psi,j)$};
        \draw[-Stealth,black!50] (-4.5,-1.4) to (-3,0.7);

        \draw [black,fill=black!10] plot [smooth cycle] coordinates {(-1.5,2.4) (-1, 2.3) (-1, 2.2) (-1.2, 2.1) (-2,2)};
        \draw [black] plot [smooth] coordinates {(-0.975, 2.285) (-0.7,2.25) (0,2.6) (0.5,2.3) (0.3, 2) (-0.3, 2.2)(-1.2, 2.1)};
        \draw node at (1.5,2.25) {\tiny $E_{k-1}^{\beta(\sigma(\c),k-1)}$};

        \draw[-Stealth] (-2.9,1.1) to [bend left] (-2.2,2.05);
        \draw node at (-2.95, 1.65) {\tiny $f$};

        \draw[-Stealth] (0.38,1.92) to [bend right] (1.5,1.3);
        \draw node at (0.6, 1.3) {\tiny $f$};
        \draw node at (1.9, 1.3) {$...$};
        \draw[-Stealth] (2.3,1.3) to [bend left] (4.45,0.05);
        \draw node at (4.1, 0.9) {\tiny $f$};
        
        \draw[-Stealth] (-7,0)--(7,0);
        \draw[Stealth-] (0,3.75)--(0,-3);

        \draw[dotted] (-6.75,-1.125)--(6.75,-1.125);
        \draw[dotted] (-6.75,-0.375)--(6.75,-0.375);
        \draw[dotted] (-6.75,0.375)--(6.75,0.375);
        \draw[dotted] (-6.75,1.125)--(6.75,1.125);
        \draw[dotted] (-6.75,1.875)--(6.75,1.875);
        \draw[dotted] (-6.75,2.625)--(6.75,2.625);

        \draw node at (-5.5,0.7) {$P_{c_0}$};
        \draw node at (-5.5,2.2) {$P_{c_1}$};
        \draw node at (-5.5,-0.8) {$P_{c_{k+j}}$};

        \centerarc[dashed](0,0)(-29:36:6)
        \draw node at (4.3,3) {$E_0$};

        \draw [black!50] node at (2,-2) {\tiny $f^j(E_{k}^{\beta(\c,k)}(\psi,j))$};
        \draw[-Stealth, black!50] (3.2,-2)--(3.8,-1.65);

        \draw [red] (4.5,-1.125) rectangle (5,-0.375);
        \draw [fill=red, opacity=0.1] (4.5,-1.125) rectangle (5,-0.375);
        \draw [red] (5.5,-1.125) -- (5.5,-0.375);
        \draw [red] (6,-1.125) -- (6,-0.375);

        \draw [red] (4,-1.125) -- (4,-0.375);

        \draw [red] (4,-1.125) -- (4.5,-1.125);
        \draw [red] (5,-1.125) -- (6,-1.125);
        \draw [red] (4,-0.375) -- (4.5,-0.375);
        \draw [red] (5,-0.375) -- (6,-0.375);

        \draw[red] plot [smooth cycle] coordinates {(-2.8,0.7) (-2.7,0.8) (-2.5,0.75) (-2.6,0.6) (-2.75,0.55)};
        \draw[fill=red, opacity=0.1] plot [smooth cycle] coordinates {(-2.8,0.7) (-2.7,0.8) (-2.5,0.75) (-2.6,0.6) (-2.75,0.55)};

        \draw[-Stealth] (4.75,-0.75) to [bend left] (-2.65,0.65);
        \draw node at (-1.4,-0.75) {\tiny $f^{-k-j}$};

        \draw [red] node at (6.6,-0.75) {\tiny $R_l^{c_{k+j}}$};

    \end{tikzpicture}
    \caption{\footnotesize Construction of the covering $\mathcal{M}^1_{\c}$.}
    \label{fig:covering M}
\end{figure}

\subsection{Estimates on the first level covering}

Let us first estimate the sum
\[\sum_{Q\in \mathcal{R}^1_{\c}} \left(\diam Q\right)^{1+\delta}.\]
The number of elements of $\mathcal{R}_{\c}^1$ inside one rectangle $R_k^{c_0}\in\mathcal{F}^0_{\c}$ can be estimated as follows. The map $f$ is univalent in $R_k^{c_0}$ and the image of $R_k^{c_0}$ is annulus with radii $e^k$ and $e^{k+1}$, therefore there are at most $2\lceil e^{k+1}\rceil<2(e^{k+1}+1)$ such elements, giving
\begin{equation}\label{eq:R_no_ele_estimate}
\#\mathcal{R}_{\c}^1\big|_{R_{k}^{c_0}} \leqslant Ae^{k},
\end{equation}
for some $A>1$, where $\#\mathcal{R}_{\c}^1\big|_{R_{k}^{c_0}}$ denotes the number of elements of $\mathcal{R}_{\c}^1$ which are subsets of ${R_{k}^{c_0}}$. 

In general, we use the notation $\mathcal{V}\big|_{V}$, where $\mathcal{V}$ is a family of sets and $V$ is a set, to denote family of all elements of $\mathcal{V}$ which are a subsets of $V$.

Every element $S\in\mathcal{R}^1_{\c}$ is a subset of $f^{-1}(R_l^{c_1})$ for some $l$, so the diameter of $S$ can be bounded in the following way
\begin{equation}\label{eq:R_diam_estimate}
\diam(S)\leqslant \diam(R_l^{c_1})\cdot\sup \Big|\Big(f^{-1}\big|_{R_l^{c_1}\cap f(R_k^{c_0})}\Big)'\Big|\leqslant\frac{\sqrt{4\pi^2+1}}{\inf\Big|\Big(f\big|_{R_k^{c_0}}\Big)'\Big|} \leqslant Be^{-k},
\end{equation}
for some $B>1$.
Combining the above two estimations we obtain
\begin{equation}\label{eq:R<e^-k}
\sum_{Q\in \mathcal{R}_{\c}^1\big|_{R_{k}^{c_0}}} \left(\diam Q\right)^{1+\delta}\leqslant A\cdot B^{1+\delta} \cdot e^{-k\delta} \leqslant Ce^{-k\delta}
\end{equation}
with $C=AB^{2}$. 
Note that constants $A$, $B$, and $C$ does not depend on $M$.
We have
\[\sum_{Q\in \mathcal{R}^1_{\c}} \left(\diam Q\right)^{1+\delta}\leqslant\sum_{k=\lfloor M \rfloor}^{+\infty} C e^{-k\delta} \leqslant C\frac{e^{-\lfloor M \rfloor\delta}}{1-e^{-\delta}}.\]
According to \eqref{ineq:<<1}, we can guarantee that with $M$ large enough, we have
\begin{equation}\label{eq:R<1}
\sum_{Q\in \mathcal{R}^1_{\c}} \left(\diam Q\right)^{1+\delta} \leqslant \frac{1}{3}.
\end{equation}
Now, let us bound the value of the expression 
\[\sum_{Q\in \mathcal{L}^1_{\c}}\left(\diam Q\right)^{1+\delta}.\]
The construction of $\mathcal{L}^1_{\c}$ refers to already constructed $\mathcal{R}^1_{\c}$ and for that reason we can use estimates \eqref{eq:R_no_ele_estimate} and \eqref{eq:R_diam_estimate}.
Let $k\geqslant \lfloor M \rfloor+1$ and $R = R_{-k}^{c_0}$. Let $F=f^{t_k+1}$. First, we prove that $F$ has bounded distortion on $R$.

\begin{lem}\label{lem:distortion}
There exists $D=D(M)>e^{\frac{\pi}{3}+1}$ such that for every $k\geqslant \lfloor M \rfloor+1$ we have
\[\frac{\sup_{z\in R}|F'(z)|}{\inf_{z\in R}|F'(z)|}\leqslant D.\]
Moreover, with taking $M$ large enough, we can get $D$ arbitrarily close to $e^{\frac{\pi}{3}+1}$.
\end{lem}
\begin{proof}
Let $\tilde{F}=f^{t_k-1}$ be the function defined on $f(R)$ and let $\varphi$ be the inverse of $\tilde{F}$. Let $r=f^{t_k-1}(0)-f^{t_k-2}(0)$. By Lemma \ref{lem:distortion_v0} we have
\begin{equation}\label{eq:distortion}
\frac{\sup_{z\in f(R)}|\tilde{F}'(z)|}{\inf_{z\in f(R)}|\tilde{F}'(z)|}\leqslant \tilde{D}(\theta) =  \frac{(1+\theta)^4}{(1-\theta)^4}.
\end{equation}
with $\theta=\frac{\pi}{6r}$.
Note that we can have arbitrarily small $\theta$ by taking $M$ large enough, because then $r$ become larger, and therefore $\tilde{D}(\theta)$ can be as close to $1$ as we want. To get estimation on the distortion of $F$ we simply use the fact that $F=f\circ \tilde{F}\circ f$ and we repeat reasoning presented in \eqref{eq:distortion_of_composition_v0}:
\begin{align}\label{eq:distortion_of_composition}
\frac{\sup_{z\in R}|F'(z)|}{\inf_{z\in R}|F'(z)|} &  \leqslant \frac{\sup_{z\in  R}|f'(z)|\cdot\sup_{z\in f(R)}|\tilde{F}'(z)|\cdot\sup_{z\in \tilde{F}(f(R))}|f'(z)|}{\sup_{z\in  R}|f'(z)|\cdot\inf_{z\in f(R)}|\tilde{F}'(z)|\cdot\inf_{z\in \tilde{F}(f(R))}|f'(z)|}\nonumber\\
& \leqslant e\cdot\tilde{D}\cdot\frac{e^{r+\frac{\pi}{6}}}{e^{r-\frac{\pi}{6}}} = \tilde{D}e^{\frac{\pi}{3}+1}.
\end{align}
Therefore $D=\tilde{D}e^{\frac{\pi}{3}+1}$ is the bound on distortion, which can be arbitrarily close to $e^{\frac{\pi}{3}+1}$, by taking $M$ large enough.
\end{proof}
\begin{cor}\label{cor:dist}
With the above estimation on distortion of $F$ on $R$ we have
    \[\frac{1}{\inf_{z\in R}|F'(z)|} \leqslant \frac{D}{\sup_{z\in R}|F'(z)|} \leqslant D \frac{\diam(R)}{\diam(F(R))}.\]
\end{cor}

From now on, let $D>0$ be the bound on the distortion of $F$. Note that $D$ can be chosen independently of $M$, i.e. enlarging $M$ does not require changing $D$.

Let here $\tilde{c}=c_{t_k + 1}$ and $\tilde{\mathcal{R}}_l = \mathcal{R}_{\sigma^{t_k+1}(\c)}^1\big|_{R_l^{\tilde{c}}}$.
By the definition of $t_k$ the set $F(R)$ is contained in both $H_{r_k}^+$ and $H_{p_k}^-$ where 
$p_k = f^{t_k}(0)\cdot e^{\frac{\pi}{6}}$.
To obtain a covering of $R$ we pull back elements of $\mathcal{R}_{\sigma^{t_k+1}(\c)}^1$ which intersect $F(R)$. There exist $l_{\min}$ and $l_{\max}$ such that all such elements of $\mathcal{R}_{\sigma^{t_k+1}(\c)}^1$ are elements of \[\bigcup_{l=l_{\min}}^{l_{\max}}\tilde{\mathcal{R}}_{l}.\]
Moreover, $l_{\min}$ and $l_{\max}$ can be chosen in such a way that
\[l_{\max} - l_{\min} - 2 \leqslant \diam F(R) \leqslant l_{\max} - l_{\min} + 2\pi,\]
simply by taking largest possible $l_{\min}$ and smallest possible $l_{\max}$. Thus, $l_{\max} - l_{\min}\leqslant E\diam F(R)$ for some $E>0$. We have $l_{\min}\geqslant k\geqslant \lfloor M \rfloor+1$ because $r_k > k \geqslant \lfloor M \rfloor+1$.

Finally, we can estimate
\begin{align}\label{eq:estimation_on_L1}
\sum_{Q\in \mathcal{L}^1_{\c}\big|_{R}} (\diam Q)^{1+\delta}
&\leqslant \sum_{l=l_{\min}}^{l_{\max}} \sum_{S\in \tilde{\mathcal{R}}_l} \bigl(\diam F^{-1}(S)\bigr)^{1+\delta} \nonumber 
\leqslant \sum_{l=l_{\min}}^{l_{\max}} A e^{l}
\bigg(
\frac{B e^{-l}}
{\inf \big| \big(F\big|_{F^{-1}(R_l^{\tilde{c}})}\big)' \big|}
\bigg)^{1+\delta} \nonumber \\
&\leqslant \sum_{l=l_{\min}}^{l_{\max}} AB^{1+\delta}e^{-\delta l} \left(\frac{D\cdot\diam(R)}{\diam(F(R))}\right)^{1+\delta} \nonumber \\
&\leqslant CD^{1+\delta}E e^{-k\delta}\frac{\diam(F(R))\cdot \diam(R) ^ {1+\delta}}{\diam(F(R)) ^ {1+\delta}}  \nonumber \\
&\leqslant CD^{1+\delta}E\cdot \sqrt{4\pi^2+1}^{1+\delta}\cdot \left(\tfrac{6}{\pi}\right)^\delta \cdot e^{-k\delta}  \leqslant\tilde{C}e^{-k\delta}.
\end{align}

First inequality follows from the fact that every considered $Q$ is subset of the preimage of some $S$ which belongs to $\tilde{\mathcal{R}}_l$ for some $l$ between $l_{\min}$ and $l_{\max}$. To obtain next inequality we use inequalities $\eqref{eq:R_no_ele_estimate}$ and $\eqref{eq:R_diam_estimate}$. In third inequality we use Corollary \ref{cor:dist}. In the fourth inequality we use the property $l_{\max} - l_{\min}\leqslant E\cdot\diam F(R)$, the fact that $l_{\min}\geqslant k$ and the definition of $C$, where we again assume that $M$ is large enough. The penultimate inequality follows from the definition of $t_k$, which implies that $\diam(F(R))\geqslant \frac{\pi}{6}$, and the fact that $\diam(R)=\sqrt{4\pi^2+1}$. In the last inequality we use $\delta < 1$, and we define the new constant
\[\tilde{C} = CD^2E(4\pi^2+1)\tfrac{6}{\pi}.\]
Again, with $M$ large enough, we have
\begin{equation}\label{eq:L<1}
\sum_{Q\in \mathcal{L}^1_{\c}} (\diam Q)^{1+\delta} \leqslant \sum_{k=\lfloor M \rfloor+1}^{+\infty} \tilde{C}e^{-k\delta} \leqslant \tilde{C}\frac{e^{-\lfloor M \rfloor\delta}}{1-e^{-\delta}} \leqslant \frac{1}{3}.
\end{equation}
Finally, we turn to the middle subfamily. Let us estimate
\[\sum_{Q\in \mathcal{M}^1_{\c}} (\diam Q)^{1+\delta},\]
which is equal to
\[\sum_{p= 1}^{+\infty}\sum_{Q\in \mathcal{M}^1_{\c}\big |_{E_p^{\beta(\c,p)}}} (\diam Q)^{1+\delta}.\]
The middle sum can be expressed as
\begin{align*}
\sum_{Q\in \mathcal{M}^1_{\c}\big |_{E_p^{\beta(\c,p)}}} (\diam Q)^{1+\delta} & = 
\sum_{Q\in \mathcal{M}^1_{\c}\big |_{E_p^{\beta(\c,p)}(+)}} (\diam Q)^{1+\delta} + \sum_{Q\in \mathcal{M}^1_{\c}\big |_{E_p^{\beta(\c,p)}(-)}} (\diam Q)^{1+\delta} \\
& \leqslant \sum_{\substack{\psi \in \{+,-\} \\ j \in \{0,1,2\}}} \sum_{Q\in \mathcal{M}^1_{\c}\big |_{E_p^{\beta(\c,p)}(\psi,j)}} (\diam Q)^{1+\delta}.
\end{align*}
For a given $\psi \in \{+,-\}$ and $j \in \{0,1,2\}$ we have
\begin{align*}
\sum_{Q\in \mathcal{M}^1_{\c}\big |_{E_p^{\beta(\c,p)}(\psi,j)}} (\diam Q)^{1+\delta} & \leqslant \sum_{k=\lfloor M \rfloor}^{+\infty} \sum_{S\in\mathcal{R}_{\sigma^{p+j}(\c)}^1\big |_{R_k^{c_{p+j}}}} (\diam(f^{-p-j}(S)))^{1+\delta} \\
& \leqslant \sum_{k=\lfloor M \rfloor}^{+\infty} AB^{1+\delta}e^{-k\delta}\frac{1}{\eta^{(p+j)(1+\delta)}}
 \leqslant \frac{C}{\eta^{p(1+\delta)}}\frac{e^{-\lfloor M \rfloor\delta}}{1-e^{-\delta}}.
\end{align*}
The first inequality follows from the fact that every considered set $Q$ is a subset of $f^{-p-j}(S)$ for $S$ being element of $\mathcal{R}_{\sigma^{p+j}(\c)}^1$ in some rectangle $R_k^{c_{p+j}}$, $k\geqslant \lfloor M \rfloor$. The second inequality follows from the \eqref{eq:R_no_ele_estimate}, \eqref{eq:R_diam_estimate} and Lemma \ref{lem:eta^k}.
Therefore, 
\[\sum_{Q\in \mathcal{M}^1_{\c}} (\diam Q)^{1+\delta}
\leqslant 6\cdot C\cdot \frac{e^{-\lfloor M \rfloor\delta}}{1-e^{-\delta}}\cdot\sum_{p=1}^{+\infty}\frac{1}{\eta^{p(1+\delta)}} \leqslant \frac{6\cdot C}{\eta^{1+\delta}-1} \cdot\frac{e^{-\lfloor M \rfloor\delta}}{1-e^{-\delta}} \leqslant \hat{C}\frac{e^{-\lfloor M \rfloor\delta}}{1-e^{-\delta}}.\]
Last inequality defines constant $\hat{C}$.

Once again, with $M$ large enough we have
\begin{equation}\label{eq:M<1}
\sum_{Q\in \mathcal{M}^1_{\c}} (\diam Q)^{1+\delta}
\leqslant \frac{1}{3}.
\end{equation}
So, from this point on we assume that $M$ is large enough that
\[\max\{C,\tilde{C}, \hat{C}\}\cdot \frac{e^{-\lfloor M \rfloor\delta}}{1-e^{-\delta}} \leqslant \frac{1}{3}.\]

\subsection{Construction of the covering at any level}
Now, we construct coverings $\mathcal{F}_{\c}^n$ at any level $n\geqslant 2$. We follow the idea from the construction of the first level covering. Again, we consider three subfamilies $\mathcal{R}^n_{\c}$, $\mathcal{L}^n_{\c}$, $\mathcal{M}^n_{\c}$ such that
$\mathcal{F}^n_{\c} = \mathcal{R}^n_{\c} \cup \mathcal{L}^n_{\c} \cup \mathcal{M}^n_{\c}$
and they are covering of $\La\cap H_M^{+}$, $\La\cap H_M^{-}$ and $\La\cap H_M^{\text{mid}}$ respectively.

For the first one, $\mathcal{R}^n_{\c}$, let $V\in\mathcal{R}^1_{\c}$. The image of $V$ is a subset of some rectangle $R_{l}^{c_1}$, by definition of $\mathcal{R}^1_{\c}$, which contains elements of covering $\mathcal{F}^{n-1}_{\sigma(\c)}$. We pull back these sets with $f^{-1}$ as in the construction of $\mathcal{R}_{\c}^1$, leading to
\[
\mathcal{R}^n_{\c} = \{V\cap f^{-1}(S): f(V)\cap S \neq\varnothing,\; V\in\mathcal{R}^1_{\c},\;S\in\mathcal{F}_{\sigma(\c)}^{n-1} \}.
\]
Recall 
$
\mathcal{F}^{n-1}_{\sigma(\c)}
=
\mathcal{R}^{n-1}_{\sigma(\c)}
\cup
\mathcal{L}^{n-1}_{\sigma(\c)}
\cup
\mathcal{M}^{n-1}_{\sigma(\c)}$ and observe that for a given $V \in \mathcal{R}^1_{\c}$ the image of $V$ usually intersects only elements of one of the families
$\mathcal{R}^{n-1}_{\sigma(\c)}$, $\mathcal{L}^{n-1}_{\sigma(\c)}$, $\mathcal{M}^{n-1}_{\sigma(\c)}$ --
since $V$ is contained in a rectangle $R_l^{c_1}$ for some $l$. Only for $l=\lfloor M\rfloor$ or $l=-\lfloor M\rfloor - 1$ the set $V$ may intersect members of two families from $\mathcal{R}^{n-1}_{\sigma(\c)}$, $\mathcal{L}^{n-1}_{\sigma(\c)}$, $\mathcal{M}^{n-1}_{\sigma(\c)}$, but this exceptional case does not affect the argument (some of the sets have to be counted twice).

Now, we need to define $\mathcal{L}^{n-1}_{\c}$. Using the notation
\[t_k=\min\left\{n\in\N: \diam(f^n(B(0,e^{-k+1})))\geqslant \tfrac{\pi}{6}\right\}\] 
we define
\begin{equation}
\mathcal{L}^n_{\c} = \{V\cap f^{-t_k-1}(S):\; f^{t_k+1}(V)\cap S\neq\varnothing,\; V\in\mathcal{L}^1_{\c} , \;S\in\mathcal{R}_{\sigma^{t_k+1}(\c)}^{n},\},
\end{equation}
where $k$ is such that $V\subseteq R_{-k}^{c_0}$, and we consider branch of $f^{-t_k-1}$ which corresponds to the finite sequence $c_0c_1\ldots c_{t_k}$. 

Finally, we define the covering $\mathcal{M}_{\c}^n$ by 
\begin{align*}
\mathcal{M}_{\c}^n = \{ V \cap E_k^{\beta(\c,k)}(\psi, j) \cap f^{-k-j}(S): & \;f^{k+j}( V\cap E_k^{\beta(\c,k)}(\psi, j))\cap S\neq\varnothing,\\
& \; V\in\mathcal{M}^1_{\c}, \; S\in\mathcal{R}_{\sigma^{k+j}(\c)}^{n},\;  \\
& \; k\geqslant 1,\; j\in\{0,1,2\},\; \psi\in\{+,-\}  \},
\end{align*}
where again $f^{-k-j}$ is a branch which corresponds to $c_0,c_1,\ldots,c_{k+j-1}$.

Note that in the definition above for any chosen $V$ there is only one set of parameters $\psi$ and $j$ for which the intersection is non-empty (i.e. we don't really choose all the possibilities for $j$ and $\psi$).

The order of the constructions is thus as follows. We start with the $\mathcal{F}_{\c}^0$ which is the family of rectangles. Using this family we construct $\mathcal{R}_{\c}^1$. Based only on $\mathcal{R}_{\c}^1$ we construct both $\mathcal{L}_{\c}^1$, $\mathcal{M}_{\c}^1$, and therefore we complete the construction of $\mathcal{F}_{\c}^1$. Then we can construct $\mathcal{R}_{\c}^2$, based on entire family $\mathcal{F}_{\c}^1$ (and $\mathcal{F}_{\c}^0$), and the inductive process of constructing coverings continues.

\subsection{Estimates on the covering at any level}
Let the inductive assumption be that for $n\geqslant 1$, every exponentially bounded $\c$, and every $k\geqslant \lfloor M \rfloor$ we have
\[\sum_{Q\in\mathcal{F}_{\c}^{n}} (\diam Q)^{1+\delta} \leqslant 1\]
and
\[\sum_{Q\in\mathcal{R}_{\c}^{n}\big|_{R_{k}^{c_0}}} (\diam Q)^{1+\delta} \leqslant Ce^{-k\delta}.\]
Notice that the first condition is satisfied for $n=1$ due to \eqref{eq:R<1}, \eqref{eq:L<1}, and \eqref{eq:M<1}, and the second condition is satisfied for $n=1$ due to \eqref{eq:R<e^-k}.
Now, let us suppose that these condition are also satisfied for some $n\geqslant 1$.

First, let us consider the rectangle $R_k^{c_0}$ for $k\geqslant \lfloor M \rfloor$. For every $V\in\mathcal{F}_{\c}^1\big|_{R_k^{c_0}} = \mathcal{R}_{\c}^1\big|_{R_k^{c_0}}$, by definition of $\mathcal{R}_{\c}^1$, we have $f(V)\subseteq R_{m}^{c_1}$ for $m=m(V)\in\Z$. Therefore 
\begin{align*}
\sum_{Q\in\mathcal{R}_{\c}^{n+1}\big|_{R_k^{c_0}}} (\diam Q)^{1+\delta} & = \sum_{V\in\mathcal{R}_{\c}^1\big|_{R_k^{c_0}}} \sum_{S\in\mathcal{F}_{\sigma(\c)}^{n}\big|_{R_m^{c_1}}} (\diam(f^{-1}(S)\cap V))^{1+\delta} \\
& \leqslant \sum_{V\in\mathcal{R}_{\c}^1\big|_{R_k^{c_0}}} \sum_{S\in\mathcal{F}_{\sigma(\c)}^{n}\big|_{R_m^{c_1}}} \bigg(\frac{\diam(S)}{\inf\big|\big(f\big|_{V}\big)'\big|}\bigg)^{1+\delta} \\
& = \sum_{V\in\mathcal{R}_{\c}^1\big|_{R_k^{c_0}}} \frac{1}{\big(\inf\big|\big(f\big|_{V}\big)'\big|\big)^{1+\delta}}\cdot\sum_{S\in\mathcal{F}_{\sigma(\c)}^{n}\big|_{R_m^{c_1}}} (\diam(S))^{1+\delta} \\
& \leqslant Ae^k \cdot (Be^{-k})^{1+\delta} \leqslant C e^{-k\delta}.
\end{align*}
In the first equality we use the fact that every considered $Q$ is, by definition of $\mathcal{R}_{\c}^n$, of the form $f^{-1}(S)\cap V$ for some $V\in\mathcal{R}_{\c}^1\big|_{R_k^{c_0}}$ and $S\in\mathcal{F}_{\c}^{n}\big|_{R_m^{c_1}}$. Next two steps, first inequality and second equality, follow from the estimation of diameter of image of set $S$ under $f^{-1}$ and reordering of terms, respectively. To obtain the second inequality we estimate the first sum using \eqref{eq:R_no_ele_estimate} and estimation on derivative of $f$ presented in \eqref{eq:R_diam_estimate}, since $V$ is a subset of $R_k^{c_0}$, and we bound the second sum by $1$ using the inductive assumption. The last step follows from the definition of $C$. We proved the second part of the inductive assertion. Using \eqref{eq:R<1} we obtain
\[
\sum_{Q\in\mathcal{R}_{\c}^{n+1}} (\diam Q)^{1+\delta} \leqslant \sum_{k=\lfloor M \rfloor}^{+\infty} C e^{-k\delta}  = C \frac{e^{-\lfloor M \rfloor\delta}}{1-e^{-\delta}} \leqslant \frac{1}{3}.
\]

Now, let us prove the same for $\mathcal{L}_{\c}^{n+1}$ and $\mathcal{M}_{\c}^{n+1}$. Let us consider the rectangle $R = R_{-k}^{c_0}$ for $k\geqslant \lfloor M \rfloor+1$. Let $F=f^{t_k + 1}$, $\tilde{\c} = \sigma^{t_k+1}(\c)$, and $\tilde{c} = c_{t_k+1}$. For every $V\in\mathcal{F}_{\c}^1\big|_{R} = \mathcal{L}_{\c}^1\big|_{R}$, by definition of $\mathcal{L}_{\c}^1$ and $t_k$, and the estimation $r_k > k$, we have 
\[f^{t_k+1}(V)\subseteq W(V) \subseteq R_{m}^{\tilde{c}}\]
for $W(V)\in\mathcal{R}_{\tilde{\c}}^1$ and $m=m(V)\geqslant k$. Moreover, $W(V)$ is injective. To construct $\mathcal{L}_{\c}^1\big|_R$ we needed to pull back with $F$ only these elements of $\mathcal{R}_{\tilde{\c}}^1$ which intersect $F(R)$. Thus we can assume that $l_{\min} \leqslant m \leqslant l_{\max}$ with $l_{\max} - l_{\min} \leqslant E\cdot\diam(F(R))$ for some $E>0$.
For every such $m$ let the family $\mathcal{V}_m$ be the set of these $V\in\mathcal{L}_{\c}^{1}\big|_{R}$ for which $f^{t_k+1}(V)\subseteq R_{m}^{\tilde{c}}$.
Now, we can perform the following estimation
\begin{align*}
\sum_{Q\in\mathcal{L}_{\c}^{n}\big|_{R}} (\diam Q)^{1+\delta} & = \sum_{V\in\mathcal{L}_{\c}^{1}\big|_{R}} \sum_{S\in\mathcal{R}_{\tilde{\c}}^{n}\big|_{W(V)}} (\diam(F^{-1}(S)\cap V))^{1+\delta}  \\
& = \sum_{m=l_{\min}}^{l_{\max}} \sum_{V\in\mathcal{V}_{m} }\sum_{S\in\mathcal{R}_{\tilde{\c}}^{n}\big|_{W(V)}} (\diam(F^{-1}(S)\cap V))^{1+\delta} \\
& \leqslant \sum_{m=l_{\min}}^{l_{\max}} \sum_{S\in\mathcal{R}_{\tilde{\c}}^{n}\big|_{R_m^{\tilde{c}}}} (\diam(F^{-1}(S)\cap R))^{1+\delta} \\
& \leqslant \sum_{m=l_{\min}}^{l_{\max}} \sum_{S\in\mathcal{R}_{\tilde{\c}}^{n}\big|_{R_m^{\tilde{c}}}} \bigg(\frac{\diam(S)}{\inf\big|\big(F\big|_{R}\big)'\big|}\bigg)^{1+\delta} \\
& = \sum_{m=l_{\min}}^{l_{\max}} \frac{1}{\left(\inf\big|\big(F\big|_{R}\big)'\big|\right)^{1+\delta}}\cdot\sum_{S\in\mathcal{R}_{\tilde{\c}}^{n}\big|_{R_m^{\tilde{c}}}} (\diam(S))^{1+\delta} \\
& \leqslant  \sum_{m=l_{\min}}^{l_{\max}} \frac{1}{\left(\inf\big|\big(F\big|_{R}\big)'\big|\right)^{1+\delta}} \cdot Ce^{-m\delta}  \leqslant  Ce^{-k\delta}\cdot \sum_{m=l_{\min}}^{l_{\max}} \frac{1}{\left(\inf\big|\big(F\big|_{R}\big)'\big|\right)^{1+\delta}} \\
& \leqslant  Ce^{-k\delta}\cdot E\cdot\diam(F(R)) \left(\frac{D\cdot\diam(R)}{\diam(F(R))}\right)^{1+\delta} \\
& \leqslant  CD^{1+\delta}E\cdot\sqrt{4\pi^2+1}^{1+\delta}\cdot\left(\tfrac{6}{\pi}\right)^{1+\delta}\cdot e^{-k\delta}   \leqslant  \tilde{C}e^{-k\delta}  \\
\end{align*}
First equality follows from the fact that every considered $Q$ is subset of the preimage of some $S$ which belongs to $\mathcal{R}_{\tilde{\c}}^1$ and is a subset of $W(V)$. Next equality and inequality follows from the definition of the subfamilies $\mathcal{V}_m$. Next, we estimate $\diam(S)$ using derivative of $F$. Then we split sums, and in the third and fourth inequalities we use inductive assumption and the fact that $m\geqslant k$. The fifth inequality follows from the fact that $l_{\max}-l_{\min} \leqslant E\cdot\diam(R)$ and Corollary \ref{cor:dist}. In the penultimate inequality we use $\diam(R)\leqslant\sqrt{4\pi^2+1}$ and $\diam(F(R))\geqslant\frac{\pi}{6}$. In last step we use the fact that $\delta<1$ and the definition of $\tilde{C}$ from \eqref{eq:estimation_on_L1}. As in the \eqref{eq:L<1} we obtain
\[\sum_{Q\in \mathcal{L}^1_{\c}} (\diam Q)^{1+\delta} \leqslant \sum_{k=\lfloor M \rfloor+1}^{+\infty} \tilde{C}e^{-k\delta} \leqslant \tilde{C}\frac{e^{-\lfloor M \rfloor\delta}}{1-e^{-\delta}} \leqslant \frac{1}{3}.\]

Finally, we prove the required estimation for $\mathcal{M}_{\c}^{n}$. For $k\geqslant 1$, $\psi\in\{+,-\}$, $j\in\{0,1,2\}$ let $F=f^{k+j}$, $\tilde{\c} = \sigma^{k+j}(\c)$, and $\tilde{c} = c_{k+j}$. For every $V\in\mathcal{F}_{\c}^1\big|_{E_{k}^{\beta(\c, k)}(\psi,j)} = \mathcal{M}_{\c}^1\big|_{E_{k}^{\beta(\c, k)}(\psi,j)}$, by definition of $\mathcal{M}_{\c}^1$, we have 
\[f^{k+j}(V)\subseteq W(V) \subseteq R_{m}^{\tilde{c}}\]
for $W(V)\in\mathcal{R}_{\tilde{\c}}^1$ and $m=m(V)\geqslant \lfloor M\rfloor$. Moreover, $W(V)$ is injective. By the proof of Lemma \ref{lem:eta^k} we have also $m\leqslant \left\lceil e^{e^{e^M}} \right\rceil= M'$. We get the following
\begin{align*}
\sum_{Q\in\mathcal{M}_{\c}^1\big|_{E_{k}^{\beta(\c, k)}(\psi,j)}} (\diam Q)^{1+\delta} & = \sum_{V\in\mathcal{M}_{\c}^1\big|_{E_{k}^{\beta(\c, k)}(\psi,j)}} \sum_{S\in\mathcal{R}_{\tilde{\c}}^{n}\big|_{W(V)}} (\diam(F^{-1}(S)\cap V))^{1+\delta} \\
& \leqslant \sum_{m=\lfloor M \rfloor}^{M'} \sum_{S\in\mathcal{R}_{\tilde{\c}}^{n}\big|_{R_m^{\tilde{c}}}} \bigg(\frac{\diam(S)}{\inf\big|\big(F\big|_{E_{k}^{\beta(\c, k)}(\psi,j)}\big)'\big|}\bigg)^{1+\delta} \\
& = \sum_{m=\lfloor M \rfloor}^{M'}  \frac{1}{\left(\inf\big|\big(F\big|_{V}\big)'\big|\right)^{1+\delta}}\cdot\sum_{S\in\mathcal{R}_{\tilde{\c}}^{n}\big|_{R_m^{\tilde{c}}}} (\diam(S))^{1+\delta} \\
& \leqslant \sum_{m=\lfloor M \rfloor}^{M'}  \frac{1}{\eta^{(k+j)(1+\delta)}}\cdot Ce^{-m\delta} \\
& \leqslant \frac{C}{\eta^{(k+j)(1+\delta)}} \cdot \frac{e^{-\lfloor M \rfloor\delta}}{1-e^{-\delta}} \leqslant \frac{C}{\eta^{k(1+\delta)}} \cdot \frac{e^{-\lfloor M \rfloor\delta}}{1-e^{-\delta}}\\
\end{align*}
First equality, first inequality, and second equality follows the logic of the analogous estimation for $\mathcal{L}_{\c}^n$. Then, we use inductive assumption, and then we use Lemma \ref{lem:eta^k} to bound the derivative of $F$. Summing over $k\geqslant 1$, $\psi\in\{+,-\}$ and $j\in\{0,1,2\}$ we obtain
\[\sum_{Q\in \mathcal{M}^1_{\c}} (\diam Q)^{1+\delta}
\leqslant 6\cdot C\cdot \frac{e^{-\lfloor M \rfloor\delta}}{1-e^{-\delta}}\cdot\sum_{p=1}^{+\infty}\frac{1}{\eta^{p(1+\delta)}} \leqslant \frac{6\cdot C}{\eta^{1+\delta}-1} \cdot \frac{e^{-\lfloor M \rfloor\delta}}{1-e^{-\delta}} \leqslant \hat{C}\frac{e^{-\lfloor M \rfloor\delta}}{1-e^{-\delta}} \leqslant \frac{1}{3}.\]
Having all three estimates we end the proof of the first part of the inductive assumption, i.e. for every $n$ we have
\[\sum_{Q\in\mathcal{F}_{\c}^{n}} (\diam Q)^{1+\delta} \leqslant 1.\]
Trivially from the construction and the estimates on the derivative we also have the second part of the assumption, 
\[\sup_{Q\in\mathcal{F}_{\c}^{n}}\diam(Q) \to 0 \mbox{\quad as $n\to +\infty$}.\]
Thus, taking $\mathcal{F}_{\c}^{n}$ as a covering of $\La$, and taking $n\to +\infty$ proves that $\dim_H(\La) \leqslant 1+\delta$. Taking $\delta \to 0$ (equivalently $M\to+\infty$) finishes the proof.
\bibliographystyle{acm}
\bibliography{biblio.bib}
\end{document}